\documentclass{article}
\usepackage{amsfonts}
\usepackage{amsmath}
\usepackage{amssymb}
\usepackage{amsthm}
\usepackage{latexsym}
\usepackage{url}
\usepackage{verbatim}
\usepackage{amsrefs}
\usepackage{titling}

\usepackage{fancyref}

\theoremstyle{plain}
\newtheorem{theorem}{Theorem}

\newtheorem{lemma}[theorem]{Lemma}

\theoremstyle{definition}
\newtheorem{definition}{\mdseries\scshape Definition}

\theoremstyle{remark}

\newcommand{\set}[1]{\left\{#1\right\}}

\newcommand{\Z}{\mathbb{Z}}
\newcommand{\Zp}{\mathbb{Z}_p}
\newcommand{\Qp}{\mathbb{Q}_p}
\newcommand{\oneunit}[1]{\left\langle#1\right\rangle}
\usepackage{wrapfig}
\usepackage[export]{adjustbox}
\usepackage{color, colortbl}

\definecolor{periwinkle}{rgb}{0.8, 0.8, 1.0}
\definecolor{mistyrose}{rgb}{1.0, 0.89, 0.88}
\definecolor{mossgreen}{rgb}{0.68, 0.87, 0.68}
\definecolor{palegreen}{rgb}{0.6, 0.98, 0.6}
\definecolor{lightgray}{rgb}{0.85,0.85,0.85}

\title{Counting Solutions to Discrete Non-Algebraic Equations Modulo Prime Powers}
\author{Abigail Mann}
\date{May 20, 2016}

\begin{document}

\begin{titlingpage}
    \maketitle
    \begin{abstract}
        As society becomes more reliant on computers, cryptographic security becomes increasingly important. Current encryption schemes include the ElGamal signature scheme, which depends on the complexity of the discrete logarithm problem. It is thought that the functions that such schemes use have inverses that are computationally intractable. In relation to this, we are interested in counting the solutions to a generalization of the discrete logarithm problem modulo a prime power.  This is achieved by interpolating to p-adic functions, and using Hensel’s lemma, or other methods in the case of singular lifting, and the Chinese Remainder Theorem.
    \end{abstract}
\end{titlingpage}

\section{Introduction}

Society has become increasingly reliant on computers for storing information and communicating securely. People expect that the cryptographic schemes currently in use will keep their information confidential and will allow them to verify the authenticity of any piece of information that they see. Public key cryptography schemes involve functions that are easy to compute one way using a publicly available key (to encrypt or verify signatures), but have inverses that are difficult to compute without a private key, so that decryption or creating a signature is only feasible for one user. Cryptographic schemes such as Diffie-Hellman key exchange and ElGamal encryption and signature schemes often use exponential modular mappings like the discrete exponentiation map $f: \Z \rightarrow \Z/p\Z$, where $x \mapsto g^x \pmod{p}$ and $g \in \Z$, $p$ a prime. These are used since they are generally believed to be computationally infeasible to invert for large prime $p$ ~\cite[Chapter 7]{trappe_wash}.

However, the security of these schemes is still being analyzed, since any insight into their structure may reveal a vulnerability. There has been previous  analysis of the maps $x \mapsto g^x \pmod{p}$ and $x \mapsto g^{x^2} \pmod{p}$ using functional graphs in \cite{lindle}, \cite{cloutier}, and \cite{wood}. Camenisch and Stadler look at the double discrete logarithm of $y$, $g^{a^x} \equiv y \pmod{c}$ as well as the $n$th root of the discrete logarithm of $z$, $g^{x^n} \equiv z \pmod{c}$, where $x,a,n,c \in \Z$, and $g,y$ are in a cyclic group $G$, for use in cryptographic signature schemes where there are multiple keys that allow for the revelation of partial information \cite{camen_stad}.

We study the $n$th roots of discrete logarithms 
 in this paper by counting integer solutions to $g^{x^n} \equiv x^k \pmod{p^e}$, where $g,n,k,e \in \Z$, $p$ is a prime, and $p \nmid g$. This may give us some insight into the structure of $n$th roots of discrete logarithms. Although it is not directly used in any cryptographic schemes today, one may be built off of this equation if its structure acts sufficiently random. The idea for counting solutions to these types of congruences was inspired by \cite{hold_rob}, which uses $p$-adic interpolation, Hensel's lemma, and the Chinese remainder theorem. This type of analysis can also be found in \cite{mann_yeoh} and \cite{waldo_zhu}, which applies these methods to the Welch Equation and the Discrete Lambert map. 

In this paper we find that for $x$ in a certain range, we can determine the exact number of solutions to $g^{x^n} \equiv x^k \pmod{p^e}$ when $p \nmid k$ and when $p=k$ and $n=1$. 

\subsection{Terminology and Background}

For this paper, we count solutions to $g^{x^n} \equiv x^k \pmod{p^e}$, where $g, n, k$, and $e$ are fixed integers, $p$ is a prime, and $p \nmid g$.
In order to count solutions to our congruence modulo $p^e$ for all positive integers $e$, we will find $p$-adic integers helpful, since each $p$-adic integer describes our solution modulo $p^e$ for all $e$.
Thus we will be using functions on the $p$-adics, or $\mathbb{Q}_p$, which are the completion of $\mathbb{Q}$ under the $p$-adic metric.
 First, we note the definition of the $p$-adic valuation of a rational number from ~\cite[Section 2.1]{gouvea}.

\begin{definition}
Fix a prime number $p \in \Z$. The $p$-adic valuation on $\Z$ is the function 
$$ v_p : \Z - \{0\} \rightarrow \mathbb{R}$$
defined as follows: for each integer $n \in \Z$, $ n \ne 0$, let $v_p(n)$ be the unique positive integer satisfying 
$$n = p^{v_p(n)}n' \mbox{  with  } p \nmid n'.$$
We extend $v_p$ to the field of rational numbers as follows: if $x = a/b \in \mathbb{Q}^{\times}$, then
$$v_p(x) = v_p(a)-v_p(b),$$
which is well-defined.
\end{definition}

We can now define the $p$-adic absolute value as follows:
 \begin{definition}
For any $x \in \mathbb{Q}$, we define the $p$-adic absolute value of $x$ by 
$$ \lvert x \rvert_p = p^{-v_p(x)}$$
if $x \ne 0$, and we set $\lvert 0 \rvert_p = 0$.
\end{definition}

The completion gives us all the rational $p$-adic numbers, while we need only to use a subset of $\Qp$. From ~\cite[Section 3.3]{gouvea}, we find that the $p$-adic integers $\Zp$ are defined as $$\Z_p = \{x \in \mathbb{Q}_p : \lvert x \rvert_p \leq 1\}.$$

Now that we have defined $\Z_p$, we
let $\mu_{p-1}$
be the set of all $(p-1)$-st roots of unity, where $\mu_{p-1}  \subseteq \Zp^\times$ by \cite[Cor. 4.5.10]{gouvea}. 
As stated in \cite[Cor. 4.5.10]{gouvea},
we can write each element of $\Zp^\times$ uniquely
as an element of  $\mu_{p-1} \times (1+p\Zp)$. 
So for each $x \in \Z_p^\times$  we write $x = \omega(x) \oneunit{x}$ for some 
$\omega(x) \in \mu_{p-1}$ and
$\oneunit{x} \in 1 + p\Zp$. 
 For odd prime $p$, this decomposition defines a character of $\Z_p^{\times}$, 
which is the surjective homomorphism
$$\omega: \Zp^\times \to \mu_{p-1}.$$ 
This character $\omega$ is called the Teichm\"uller character \cite[Section 4.5]{gouvea}.
We will use the factorization of $x$ into $\omega(x)\oneunit{x}$ to aid in our analysis.

Additionally, we will need the $p$-adic exponential and logarithm functions. As in $\mathbb{R}$, we can define the $p$-adic exponential and logarithm functions on certain subsets of the $p$-adic numbers as formal power series:
$$\exp_p(x) = \sum\limits_{n=0}^{\infty}\frac{x^n}{n!},
$$
$$ \log_p(1+x) = \sum\limits_{n=1}^{\infty}\frac{(-1)^{n+1}x^n}{n}.$$
These functions have radii of convergence $\lvert x \rvert_p < p^{-1/(p-1)}$ and $1$, respectively.

It is important to note that the identities $\exp_p(\log_p(1+x))=1+x$ and $\log_p(\exp_p(x))=x$ hold formally, and will also hold functionally when we have convergence. For more on these functions, see ~\cite[Section 4.5]{gouvea}.

Lastly, we will want to use a generalization of Hensel's lemma, which allows the lifting of solutions to congruences modulo $p$ to solutions modulo $p^e$, that applies to the $p$-adics. First, we will need to define a restricted power series. A formal power series is an object of the form $\sum_{i=0}^{\infty}a_ix^i$, where the $a_i$ are unrestricted coefficients, and addition and multiplication are performed similarly to polynomial operations. A restricted power series is a formal power series where $\lim_{i \rightarrow \infty}a_i = 0$. Now we can take this theorem from \cite[Cor. 3.3]{hold_rob}. 
\begin{theorem}\label{Hensel}
Let $f(x)$ be a restricted power series in $\Z_p[[x]]$ and $a$ be in $\Z_p$ such that $\frac{df}{dx}(a)$ is in $\Z_p^{\times}$ and $f(a) \equiv 0 \pmod{p}$. then there exists a unique $x \in \Z_p$ for which $x \equiv a \pmod{p}$ and $f(x) = 0$ in $\Z_p$.
\end{theorem}

With this knowledge in mind, we can now start our analysis. For this paper, we let $n,k,$ and $e$ be integers, $p$ a prime, and $g$ a unit modulo $p$ (i.e. $p\nmid g$, so $g$ has an inverse modulo $p$). We will be counting the integer solutions of the congruence $g^{x^n} \equiv x^k \pmod{p^e}$, or equivalently, the zeros of $f:\Z \rightarrow \Z/p^e\Z$, where $f(x) = g^{x^n} - x^k \pmod{p^e}$. We denote the multiplicative order of $g$ modulo $p$ as $m$.

\section{Periodicity}

The first thing to note about our function $f$ is that it is periodic, since it will restrict the range of $x$ to examine when counting solutions. The theorem in this section describes its periodicity.

\begin{lemma}\label{thm:repetition}$g^{m\cdot p^{e-1}} \equiv 1 \pmod{p^e}$.
\end{lemma}
This lemma is obtained from the proof of ~\cite[Theorem 1]{mann_yeoh}, and allows us to conclude with the following theorem.

\begin{theorem} \label{thm:xrep}
Fixing all variables except $x$, we have that
$$g^{(x+mp^e)^n}-(x+mp^e)^k  \equiv g^{x^n}-x^k \pmod{p^e}.$$
In other words, $f(x) = f(x+mp^e)$.
\end{theorem}

\begin{proof}

First, consider $g^{(x+mp^e)^n} \pmod{p^e}$.
We know $$(x+mp^e)^n = \sum_{i=0}^{n}\binom{n}{i}x^i(mp^e)^{n-i}.$$
Since $mp^{e-1}\vert mp^e$ and $mp^e$ divides all terms except $x^n$,  by Lemma~\ref{thm:repetition} we have
$$g^{(x+mp^e)^n} \equiv g^{x^n} \pmod{p^e}.$$

Now consider $(x+mp^e)^k$. We can also expand this to 
$$(x+mp^e)^k = \sum_{i=0}^{k}\binom{k}{i}x^i(mp^e)^{k-i}.$$
Since $p^e \vert mp^e$ and $mp^e$ divides all terms except $x^k$, we have
$$(x+mp^e)^k \equiv x^k \pmod{p^e}.$$

Thus $g^{(x+mp^e)^n}-(x+mp^e)^k  \equiv g^{x^n}-x^k \pmod{p^e}$.

\end{proof}


\section{Interpolation}

 Since we would like to analyze our equation $p$-adically, our first goal is to interpolate our function $f:\Z \rightarrow \Z/p^e\Z$, $f(x) = g^{x^n} -x^k \pmod{p^e}$ to a function from $\Z_p$ to $\Z_p$.

We find that although we cannot interpolate to a single continuous $p$-adic function, we can interpolate to a finite number of $p$-adic functions that agree with $f(x)$ on certain values of $x$.

\begin{theorem}
\label{gouvea-interp}
For $p \neq 2$, let $g \in \Zp^\times$ and $x_0 \in \Z/(p-1)\Z$, and let
$$I_{x_0} = \set{x \in \Z \mid x \equiv x_0
  \pmod{p-1}} \subseteq \Z.$$
Then $$F_{x_0}(x) = \omega(g)^{x_0^n}\oneunit{g}^{x^n}$$ defines a uniformly continuous function
on $\Zp$ such that $F_{x_0}(x) = g^{x^n}$ whenever $x\in I_{x_0}$.
\end{theorem}

\begin{proof}
By ~\cite[Proposition 4.6.1]{gouvea}, we need
 $I_{x_0}$ to be dense in $\Z_p$ and for each $F_{x_0}(x)$ be uniformly continuous and bounded.
 We know that if a function $f:\Z_p \rightarrow \mathbb{Q}_p$ is continuous on $\Z_p$, then it is also uniformly continuous and bounded ~\cite[Theorem 4.1.4]{katok}.
Thus, it suffices to show density of $I_{x_0}$, continuity of each $F_{x_0}$ as a function on $I_{x_0}$, and that $F_{x_0}(x)=g^{x^n}$ with the proper conditions on $x$.

We first need to prove density of $I_{x_0}$ in $\Z_p$. This is shown in the proof of ~\cite[Theorem 16]{mann_yeoh} when we let $c=1$.

Now we must show each $F_{x_0}(x) = \omega(g)^{x_0^n}\oneunit{g}^{x^n}$ is uniformly continuous on $I_{x_0}$.
Given $\epsilon>0$, find $N$ such that $p^{-N}<\epsilon$.
 Now if $x,y \in I_{x_0}$ such that 
$$\lvert x-y \rvert_p \leq p^{-N} < p^{-(N-1)} = \delta,$$ then $x=y+p^NA$ for some $A \in \Z$.
 Consider
\begin{eqnarray*}
\lvert\oneunit{g}^{x^n}-\oneunit{g}^{y^n}\rvert_p = \lvert\oneunit{g}^{(y+p^NA)^n}-\oneunit{g}^{y^n}\rvert_p &=&
 \lvert\oneunit{g}^{y^n}\rvert_p \lvert\oneunit{g}^{(y+p^NA)^n-y^n}-1\rvert_p  \\
&=& \lvert\oneunit{g}^{(y+p^NA)^n-y^n}-1\rvert_p,
\end{eqnarray*}
and using the binomial theorem, we get
$$\oneunit{g}^{(y+p^NA)^n-y^n} = \oneunit{g}^{\sum_{i=1}^{n}\binom{n}{i}y^{n-i}(p^NA)^i}.$$
If we factor out $p^N$ from the exponent, we get 
$$\oneunit{g}^{\sum_{i=1}^{n}\binom{n}{i}y^{n-i}(p^NA)^i} = \oneunit{g}^{p^Nb},$$
where $b=\sum_{i=1}^{n}\binom{n}{i}y^{n-i}(p^N)^{i-1}A^i$, which is an integer.
So we have $$\lvert\oneunit{g}^{x^n}-\oneunit{g}^{y^n}\rvert_p  = \lvert\oneunit{g}^{(y+p^NA)^n-y^n}-1\rvert_p 
= \lvert\oneunit{g}^{p^Nb}-1\rvert_p $$
Using the binomial theorem again, and the fact that $\oneunit{g} = 1+pM$, we get
$$(1+pM)^{p^Nb} = 1+p^NbpM+\frac{p^Nb(p^Nb-1)}{2}(pM)^2+\ldots+(pM)^{p^Nb}.$$
Because all terms except for the first are in $p^{N+1}\Z_p$, we see that 
$$\lvert\oneunit{g}^{p^NA}-1\rvert_p \leq p^{-(N+1)} < p^{-N} < \epsilon.$$
So the function mapping $x \rightarrow \oneunit{g}^{x^n}$ is uniformly continuous on $I_{x_0}$ and hence on $\Z_p$ by ~\cite[Thm 4.15]{katok}.
Since each $F_{x_0}(x) = \omega(g)^{x_0^n}\oneunit{g}^{x^n}$
for fixed $x_0$, and $g$,  and $ \omega(g)^{x_0^n}$ is a constant, we have that $F_{x_0}(x)$ is a constant 
times a uniformly continuous function. Hence, each $F_{x_0}(x)$ is uniformly continuous on $\Z_p$ ~\cite[Exercise 89]{katok}.

Lastly, we show that $F_{x_0}(x) = g^{x^n}$ when $x \in I_{x_0}$.
Since $x \equiv x_0 \pmod{p-1}$, we have that
 $$g^{x^n} = \omega(g)^{x^n}\oneunit{g}^{x^n} = \omega(g)^{x_0^n}\oneunit{g}^{x^n} = F_{x_0}(x).$$

\end{proof}

We can extend this theorem to multiples of the order of $g$ modulo $p$:

\begin{theorem}\label{interp}
  For this theorem only, we let $m$ be any multiple of the multiplicative order of $g$ modulo
  $p$, $p \neq 2$, so that $m \mid p-1$.
  Let $g \in \Zp^\times$ and $x_0 \in \Z/m\Z$, and let
$$J_{x_0} = \set{x \in \Z \mid x \equiv x_0
  \pmod{m}} \subseteq \Z.$$
Then $$F_{x_0}(x) = \omega(g)^{x_0^n}\oneunit{g}^{x^n}$$ defines a uniformly continuous function
on $\Zp$ such that $F_{x_0}(x) = g^{x^n}$ whenever $x\in J_{x_0}$.
\end{theorem}

\begin{proof}
Since $g^m \equiv 1 \pmod{p}$, $\omega(g)^{m^n} = \omega(g^{m^n}) = \omega(1) = 1.$ 
If $x_0, x_0' \in \Z/(p-1)\Z$ and $ x_0 \equiv x_0' \pmod{m}$,
then the two functions $F_{x_0}$ and $F_{x_0'}$ given by
Theorem~\ref{gouvea-interp} are equal and are the same as $g^{x^n}$ when
$x \in I_{x_0} \cup  I_{x_0'} \subseteq J_{x_0}$.
\end{proof}



\section{Counting Solutions}

Now that we have our $p$-adic functions, we can use those to begin counting solutions.  We begin by counting solutions to our modified congruences modulo $p$, and then proceed by lifting these solutions to $p$-adic solutions modulo $p^e$. Lastly, we will refer back to our theorems on interpolation to find when the solutions to our modified congruences will give us solutions to our original congruence $g^{x^n} \equiv x^k \pmod{p^e}$.

The following lemma analyzes solutions modulo $p$.

\begin{lemma}\label{countmodp}
Consider the equation
$$g^{x_0^n} \equiv x^k \pmod{p}.$$
Define $d = \frac{\gcd(k,p-1)}{\gcd(k,\frac{p-1}{m})}$, and let
$q_1^{\alpha_1}q_2^{\alpha_2}\cdots q_i^{\alpha_i}$ be the prime factorization of $d$.
Then there are $N = \frac{m\cdot\gcd(k,p-1)}{q_1^{\lceil{\frac{\alpha_1}{n}}\rceil}q_2^{\lceil{\frac{\alpha_2}{n}}\rceil}\cdots q_i^{\lceil{\frac{\alpha_i}{n}}\rceil}}$
solution pairs $(x_0,x)$ to the above equation, where $x_0 \in \{0,1,\ldots,m-1\}$ and $x \in \{0,1,\ldots,p-1\}$.

\end{lemma}

\begin{proof}
Let $h$ be a primitive root modulo $p$, so we can express $g\equiv h^a \pmod{p}$ and $x \equiv h^b \pmod{p}$.
So $g^{x_0^n} \equiv x^k \pmod{p}$ becomes $(h^a)^{x_0^n} \equiv (h^b)^k \pmod{p}.$
Since $h$ is a primitive root, we have that $ax_0^n \equiv bk \pmod{p-1}$. From ~\cite[Theorem 5.1]{andrews},
we have that there are $\gcd(k,p-1)$ mutually incongruent solutions for $b$ (which correspond to a distinct values of $x$)
if $\gcd(k,p-1)\mid ax_0^n$, and no solutions otherwise. So we must now count $x_0$ where $\gcd(k,p-1)\mid ax_0^n$.

We have that $\gcd(k,p-1)\mid ax_0^n$ if and only if $\frac{\gcd(k,p-1)}{\gcd(k,p-1,a)}\mid \frac{a}{\gcd(k,p-1,a)}x_0^n$. Note that
$\gcd(k,p-1,a) = \gcd(\gcd(k,p-1),a)$ so $\frac{\gcd(k,p-1)}{\gcd(k,p-1,a)}$ is relatively prime to $\frac{a}{\gcd(k,p-1,a)}$.
Now we only need to count $x_0$ that satisfy $\frac{\gcd(k,p-1)}{\gcd(k,p-1,a)}\mid x_0^n.$

Because we defined 
 $h$ and $a$ so that $g\equiv h^a \pmod{p}$, and $g$ has order $m$, we know that $\gcd(a,p-1)=\frac{p-1}{m}$. So 
$\gcd(k,p-1,a) = gcd(k,\gcd(a,p-1)) = \gcd(k,\frac{p-1}{m})$. 

Now we are left with counting $x_0$ that satisfy $\frac{\gcd(k,p-1)}{\gcd(k,\frac{p-1}{m})}\mid x_0^n$, which is the same as $d \mid x_0^n$.
In order to count the number of solutions, we look at the prime factorization of $d$. We have that 
$$ q_1^{\alpha_1}q_2^{\alpha_2}\cdots q_i^{\alpha_i} \mid x_0^n \mbox{ if and only if } q_1^{\lceil{\frac{\alpha_1}{n}}\rceil}q_2^{\lceil{\frac{\alpha_2}{n}}\rceil}\cdots q_i^{\lceil{\frac{\alpha_i}{n}}\rceil} \mid x_0,$$ 
and thus we have 
$\frac{m}{q_1^{\lceil{\frac{\alpha_1}{n}}\rceil}q_2^{\lceil{\frac{\alpha_2}{n}}\rceil}\cdots q_i^{\lceil{\frac{\alpha_i}{n}}\rceil}}$ distinct $x_0 \in \{0,1,\ldots, m\}$ that 
satisfy our conditions. Since there are $\gcd(k,p-1)$ solutions $x \in \{0,1,\ldots, p-1\}$ for each $x_0$, we have a total of 
$\frac{m\cdot\gcd(k,p-1)}{q_1^{\lceil{\frac{\alpha_1}{n}}\rceil}q_2^{\lceil{\frac{\alpha_2}{n}}\rceil}\cdots q_i^{\lceil{\frac{\alpha_i}{n}}\rceil}}$
solution pairs $(x_0,x)$ to $g^{x_0^n} \equiv x^k \pmod{p}.$

\end{proof}

\subsection{Counting solutions when $p\nmid k$}

When we lift the solutions we found modulo $p$ to solutions modulo $p^e$, we have to use different methods for when $p \nmid k$ than when $p \mid k$. We will be able to use Hensel's lemma to lift to solutions modulo $p^e$ when $p \nmid k$. The following lemma describes the result.

\begin{lemma} \label{lifting} 
For $p \neq 2$, $p \nmid k$, let $g \in \Zp^\times$ be fixed, and $x_0 \in \{0,1,\cdots,m-1\}$. If $a$ is a solution in $\{0, 1, \ldots, p-1\}$  to
$$\omega(g)^{x_0^n} \equiv g^{x_0^n} \equiv x^k \pmod{p}.$$
Then there is a unique solution in $\Zp$ to the equation 
$$\omega(g)^{x_0^n} \oneunit{g}^{x^n} = x^k $$
where $x \equiv a \pmod{p}$.
\end{lemma}

\begin{proof}
Since
  $\oneunit{g}$ is in $1+p\Zp$, we get
\begin{eqnarray*}
\oneunit{g}^{x^n} &=& (\exp_p(x^n \log_p(\oneunit{g}))  \\
&=&   (1+x^n\log_p(\oneunit{g})+
x^{2n}\log_p( \oneunit{g})^2/2! \\
& & + \mbox{higher order terms in powers of }
\log_p(\oneunit{g})),
\end{eqnarray*}
where from ~\cite[Proposition 4.5.9]{gouvea},  we know that 
$\log_p(\oneunit{g}) \in p\Zp$. Now that we have a convergent
power series since $|\log_p(\oneunit{g})^i/i!|_p \to 0$ as $i \to \infty$ ~\cite[Chapter 2, Theorem 3.1]{bachman}, we 
examine $f(x)=F_{x_0}(x)-x$ and its derivative to see if we can apply a generalization of Hensel's lemma.

Consider
\begin{eqnarray*}
f(x)=\omega(g)^{x_0^n}(1&+&x^n\log_p(\oneunit{g})+
x^{2n}\log_p( \oneunit{g})^2/2! \\
&+& \mbox{higher order terms in powers of }
\log_p(\oneunit{g})) - x^k.
\end{eqnarray*}

Since we know $\log_p{(\oneunit{g})}\in p\Zp$, so $\log_p{(\oneunit{g})} \equiv 0 \pmod{p}$, we have that
\begin{eqnarray*}
f(a)\equiv\omega(g)^{x_0^n}(1&+&a^n(0)+
a^{2n}(0) \\
&+& \mbox{higher order terms congruent to 0 (mod p) }
) -  a^k \pmod{p}
\end{eqnarray*}
\begin{eqnarray*}
\equiv\omega(g)^{x_0^n}-a^k
\equiv 0 \pmod{p}.
\end{eqnarray*}

Additionally, we have that
\begin{eqnarray*}
f'(x)=\omega(g)^{x_0^n}(nx^{n-1}\log_p{(\oneunit{g})}+(2n)x^{2n-1}\log_p{(\oneunit{g})}^2/2! \\
+{3n}x^{3n-1}\log_p{(\oneunit{g})^3/3!}+...)-ka^{k-1}
\end{eqnarray*}
so that
\begin{eqnarray*}
f'(a)&\equiv&\omega(g)^{x_0^n}(na^{n-1}(0)+(2n)a^{2n-1}(0)^2/2!+{3n}a^{3n-1}(0)^3/3!+...)-ka^{k-1}\\
&\equiv& 0-ka^{k-1} \pmod{p}.
\end{eqnarray*}
We know $a^k \equiv \omega(g)^{x_0^n} \pmod{p}$ so then
we know $a^k \not \equiv 0 \pmod{p}$ and thus $a^{k-1} \not \equiv 0 \pmod{p}$.
Also, we have $p \nmid k$.
So then $-ka^{k-1} \not \equiv{0} \pmod{p}$.
Now we know we can apply Theorem \ref{Hensel}, 
 which states that there is a unique $x \in \Z_p$ for which $x \equiv a \pmod{p}$ and $f(x)=0$ in $\Z_p$.

\end{proof}

Now that we have found solutions to our modified equations, we need to be able to piece them together to give us solutions to our original equation. The following theorem uses the results from our lemmas to give us the number of solutions to $g^{x^n} \equiv x^k \pmod{p^e}$ when $p \nmid k$.

\begin{theorem}\label{pnmidk_final}
For $p \not = 2$, let $g \in \Zp^\times$ and $n,k \in \Z$ be fixed and $p \nmid k$. Then there are  $N = \frac{m\cdot\gcd(k,p-1)}{q_1^{\lceil{\frac{\alpha_1}{n}}\rceil}q_2^{\lceil{\frac{\alpha_2}{n}}\rceil}\cdots q_i^{\lceil{\frac{\alpha_i}{n}}\rceil}}$ solutions $x$ to the equation 
$$g^{x^n} \equiv x^k \pmod{p^e}$$
for $x \in \{1,2,\cdots,p^em\}.$ 
\end{theorem}

\begin{proof}

We begin by considering the number of solutions modulo $p$ to a slightly different equation. By Lemma \ref{countmodp}, we have $\frac{m\cdot\gcd(k,p-1)}{q_1^{\lceil{\frac{\alpha_1}{n}}\rceil}q_2^{\lceil{\frac{\alpha_2}{n}}\rceil}\cdots q_i^{\lceil{\frac{\alpha_i}{n}}\rceil}}$ solution pairs $(x_0,x_1)$ to 
$g^{x_0^n} \equiv x_1^k \pmod{p}$ where the $x_0$ are distinct $\pmod{m}$ and $x_1$ are distinct $\pmod{p}$. For each $x_1$ that appears in a solution pair to $g^{x_0^n} \equiv x_1^k \pmod{p}$, then by Lemma \ref{lifting} we have a unique solution $x'$ in $\Zp$ to $\omega(g)^{x_0^n} \oneunit{g}^{(x')^n} = (x')^k$ where $x' \equiv x \pmod{p}$.
By the Chinese Remainder Theorem, we have that there is exactly one $x \in \Z/mp^e\Z$ where $x \equiv x_0 \pmod{m}$ and $x \equiv x' \pmod{p^e}$. Thus by Theorem \ref{interp} we have exactly one solution to $g^{x^n} \equiv x^k \pmod{p^e}$ in $\Z/mp^e\Z$ for every solution pair to $g^{x_0^n} \equiv x^k \pmod{p}$, and therefore there are  $\frac{m\cdot\gcd(k,p-1)}{q_1^{\lceil{\frac{\alpha_1}{n}}\rceil}q_2^{\lceil{\frac{\alpha_2}{n}}\rceil}\cdots q_i^{\lceil{\frac{\alpha_i}{n}}\rceil}}$ solutions in $\Z/mp^e\Z$ to the equation $g^{x^n} \equiv x^k \pmod{p^e}$.

\end{proof}

We find that this theorem is consistent with our results. For example, looking at $g^{x^n} \equiv x^k \pmod{7^e}$ for $0 \le x \le m\cdot 7^e$, we get the following number of solutions for all $n$ and $e$.

\begin{table}[h]\label{table1}
\caption{$g^{x^n} \equiv x^k \pmod{7^e}$ for $0\le x < m\cdot7^e$}
\begin{tabular}{| p{6mm} p{6mm} | p{2.4cm} p{2.4cm} p{2.4cm} p{2.4cm}|}
\hline
\textbf{g} & \textbf{m} & \textbf{\# solns: k=1} & \textbf{\# solns: k=2} & \textbf{\# solns: k=3} & \textbf{\# solns: k=4}  \\
\hline
\textbf{1} & \textbf{1} & 1 & \cellcolor{lightgray} 2 & \cellcolor{lightgray} 3 & \cellcolor{lightgray} 2 \\
\textbf{2} & \textbf{3} & 3 & \cellcolor{lightgray} 6 & 3 & \cellcolor{lightgray} 6\\
\textbf{3} & \textbf{6}& 6  & 6 & 6 & 6\\
\textbf{4} & \textbf{3} & 3 & \cellcolor{lightgray} 6 & 3 & \cellcolor{lightgray} 6\\
\textbf{5} & \textbf{6} & 6  & 6 & 6 & 6\\
\textbf{6} & \textbf{2} & 2 & 2 & \cellcolor{lightgray} 6 & 2\\
\hline
\end{tabular}

\end{table}
So if we look at the case when $k=4$, we find that
$$d=\frac{\gcd(k,p-1)}{\gcd(k,\frac{p-1}{m})} = \frac{\gcd(4,7-1)}{\gcd(2,\frac{7-1}{m})} = \frac{2}{\gcd(4,\frac{6}{m})} = 
\begin{cases}
1 &  \mbox{if } 2\nmid m \\
2 &  \mbox{if } 2 \mid m
\end{cases},
$$ 
and
$$N=
\begin{cases}
\frac{m\gcd(4,7-1)}{1} = 2m & \mbox{if } 2\nmid m \\
\frac{m\gcd(4,7-1)}{2} = m & \mbox{if } 2\mid m 
\end{cases},
$$
which matches the findings in Table \ref{table1}.


\subsection{Counting solutions when $p=k$ and $n=1$}
Our findings for when $p=k$ differs from our results when $p \nmid k$. For example, when $p=11$, we find the number of solutions detailed in Table \ref{table2}. We see that our $N$ solutions modulo $p$ lift to different numbers of solutions modulo $p^e$ than in the $p \nmid k$ case. This suggests that we must lift solutions modulo $p$ to solutions modulo $p^e$ differently: we will end up using induction on $e$. So, we will count solutions modulo $p^2$ and use that as the base case in our induction.

\begin{table}[h]\label{table2}
\begin{tabular}{| p{7mm} p{7mm} | p{1.5cm} p{1.5cm} p{1.5cm} p{1.5cm}|}
\hline
\textbf{g} & \textbf{m} & \textbf{\# solns: e=1} & \textbf{\# solns: e=2} & \textbf{\# solns: e=3} & \textbf{\# solns: e=4}  \\
\hline
\rowcolor{lightgray}
\textbf{1} & \textbf{1} & 1 &  11 & 11 & 11  \\
\textbf{2} & \textbf{10} &10 & 0 & 0 & 0\\
\rowcolor{lightgray}
\textbf{3} & \textbf{5} & 5 & 55 & 55 & 55\\
\textbf{4} & \textbf{5} & 5 & 0 & 0 & 0\\
\textbf{5} & \textbf{5} &  5 & 0 & 0 & 0\\
\textbf{6} & \textbf{10} &10 & 0 & 0 & 0\\
\textbf{7} & \textbf{10} &10 & 0 & 0 & 0\\
\textbf{8} & \textbf{10} &10 & 0 & 0 & 0\\
\rowcolor{lightgray}
\textbf{9} & \textbf{5} & 5 & 55 & 55 & 55\\
\textbf{10} & \textbf{2} & 2 & 0 & 0 & 0 \\
\hline
\end{tabular}
\caption{$g^x \equiv x^{11} \pmod{11^e}$ for $0\le x < m\cdot11^e$}
\end{table}

As we lift, we find that the value of $g^{p-1}$ modulo $p^2$ is important.
By Fermat's Little Theorem, we have for prime $p$ and $p \nmid g$, that $g^{p-1} \equiv 1 \pmod{p}$. Looking at this equivalence modulo $p^2$ gives the following definition.

\begin{definition}
An integer $g$ is called a Wieferich base modulo $p$ if $g^{p-1} \equiv 1 \pmod{p^2}$.
\end{definition}

Now we are able to count solutions to $g^x \equiv x^p \pmod{p^2}$, seeing that the result depends heavily on whether $g$ is a Wieferich base modulo $p$.

\begin{lemma}\label{liftp2}
Let $p \not = 2$, let $a_0$ be a solution to $g^x \equiv x^p \pmod{p}$, and let $x_0 \equiv a_0 \pmod{m}$.
Then the following are equivalent:
\begin{enumerate}
\item $g = \omega(g)\oneunit{g}$, where $\oneunit{g} \equiv 1 \pmod{p^2}$. \label{sg1is0}
\item $g$ is a Wieferich base modulo $p$ \label{sWieferich}
\item  $a_0$ lifts to at least one solution $a \in \Z/p^2\Z$ to $g^{x} \equiv x^p \pmod{p^2}$ where $a \equiv a_0 \pmod{p}$ and $a \equiv x_0 \pmod{m}$.  \label{sLifts}
\end{enumerate}
Furthermore, in \ref{sLifts}, we also have that if $a_0$ lifts to a solution in $\Z/p^2\Z$, it lifts to $p$ distinct solutions in $\Z/p^2\Z$.

\end{lemma}

\begin{proof}
For this proof, we begin by finding a congruence that holds exactly when we have a solution $a$ to $g^x \equiv x^p \pmod{p^2}$ that satisfies the conditions that $a \equiv x_0 \pmod{m}$ and $a \equiv a_0 \pmod{p}$. We will then use the equivalent statement to prove \ref{sLifts} $\implies$ \ref{sWieferich} and 
\ref{sg1is0} $\implies$ \ref{sLifts}, and then finish by showing \ref{sWieferich} $\implies$ \ref{sg1is0}.

Let $a \equiv x_0 \pmod{m}$ and $a \equiv a_0 \pmod{p}$, and consider when $0 \equiv g^a - a^p \pmod{p^2}$. 

Recall from Theorem~\ref{interp} that $ g^a = \omega(g)^{x_0}\oneunit{g}^a$ when $a \equiv x_0 \pmod{m}$.  
So since $a \equiv x_0 \pmod{m}$, we have
\begin{eqnarray}
0 \equiv g^{a} - a^p &\equiv& \omega(g)^{x_0}\oneunit{g}^a - a^p  \nonumber \\
 &\equiv& \omega(g)^{x_0}\sum_{i=0}^{\infty}\Big ( \frac{a^i(\log_p\oneunit{g})^i}{i!}\Big )-a^p \pmod{p^2} .\label{eq:1}
\end{eqnarray}
We have $a \equiv a_0 \pmod{p}$, so we get $a \equiv a_0 + a_1p \pmod{p^2}$, and thus equation (\ref{eq:1}) holds exactly when we get
\begin{eqnarray*}
\omega(g)^{x_0}\sum_{i=0}^{\infty}\Big ( \frac{(a_0+a_1p)^i(\log_p\oneunit{g})^i}{i!}\Big )-(a_0+a_1p)^p &\equiv& 0 \pmod{p^2}.
\end{eqnarray*}
Since $\log_p{\oneunit{g}} \in p\Zp$, this reduces to
\begin{eqnarray}
\omega(g)^{x_0}(1+a_0\log_p{\oneunit{g}})-(a_0+a_1p)^p &\equiv& 0 \pmod{p^2}.
\end{eqnarray}
When we expand the term $(a_0+a_1p)^p$ modulo $p^2$, we find that it is congruent to $a_0^p$, and we obtain
\begin{eqnarray}
\omega(g)^{x_0}(1+a_0\log_p{\oneunit{g}})-a_0^p &\equiv& 0 \pmod{p^2}. \label{eqn:3}
\end{eqnarray}
Note that 
$$ \log_p{\oneunit{g}} = \sum_{i=0}^{\infty}\Big ( (-1)^{i+1} \frac{(\oneunit{g}-1)^i}{i} \Big ),$$
and since $\oneunit{g}-1 \in p\Zp$, we have 
$$ \log_p{\oneunit{g}} \equiv \oneunit{g}-1 \pmod{p^2}.$$
So then we can replace $\log_p{\oneunit{g}}$ in equation (\ref{eqn:3}) to obtain
\begin{eqnarray}
\omega(g)^{x_0}(1&+&a_0(\oneunit{g}-1))-a_0^p \nonumber \\
&\equiv& \omega(g)^{x_0}a_0(\oneunit{g}-1)+\omega(g)^{x_0}-a_0^p \equiv 0 \pmod{p^2}. \label{eqn:4}
\end{eqnarray}
We have that $\oneunit{g}-1 \equiv 0 \pmod{p}$. We also know
\begin{equation}
g^{a_0} - a_0^p \equiv \omega(g)^{a_0}\oneunit{g}^{a_0}-a_0^p \equiv \omega(g)^{x_0} - a_0^p \equiv 0 \pmod{p},
\end{equation}\label{eqn: 5} 
so we can write equation (\ref{eqn:4}) as
\begin{eqnarray}
\frac{\omega(g)^{x_0}a_0(\oneunit{g}-1)}{p}+\frac{\omega(g)^{x_0}-a_0^p}{p} &\equiv& 0 \pmod{p}. \label{eqn:6}
\end{eqnarray}

Now consider the following: by (\ref{eqn: 5}) and Fermat's Little Theorem we have
$$\omega(g)^{x_0} - a_0^p \equiv \omega(g)^{x_0} - a_0 \equiv 0 \pmod{p}.$$
By definition, this is true exactly when
$$\omega(g)^{x_0} - a_0^p = rp \mbox{, for some } r \in \Z_p.$$
Now this is true if and only if
$$(\omega(g)^{x_0})^{p-1} = (a_0^p+rp)^{p-1}.$$
Since $\omega(g)$ is a $(p-1)$th root of unity, we find that $(\omega(g)^{x_0})^{p-1} = 1$. By using this fact and expanding $ (a_0^p+rp)^{p-1}$, we find that the above is equivalent to
$$ 1 = \sum_{i=1}^{p-1}\Big ( \binom{p-1}{i}(rp)^i(a_0^p)^{p-1-i}\Big ).$$
Examining this equation modulo $p^2$, we find that
$$1 \equiv a_0^{p(p-1)}+p(p-1)ra_0^{p(p-2)} \pmod{p^2},$$
and since the order of the group $\Z/p^2\Z$ is $p(p-1)$, we get $a_0^{p(p-1)} \equiv 0 \pmod{p^2}$, so
$$1 \equiv a_0^{p(p-1)}+p(p-1)ra_0^{p(p-2)} \pmod{p^2},$$
and thus
$$0 \equiv (p-1)ra_0^{p(p-2)} \pmod{p}.$$
Since $p-1, a_0 \not\equiv 0 \pmod{p}$ (if $a_0 \equiv 0 \pmod{p}$, then $1 \equiv g^0 \equiv 0^p \equiv 0 \pmod{p}$), we must have $r \equiv 0 \pmod{p}$. So then 
$\omega(g)^{x_0} - a_0^p = (kp)p  \equiv 0 \pmod{p^2}$ for some $k \in \Z$.

Thus we always have that $p \mid \frac{\omega(g)^{x_0}-a_0^p}{p}$, and we can reduce equation (\ref{eqn:6}) further:
\begin{eqnarray}
\frac{\omega(g)^{x_0}a_0(\oneunit{g}-1)}{p}+\frac{\omega(g)^{x_0}-a_0^p}{p}  
\equiv \frac{\omega(g)^{x_0}a_0(\oneunit{g}-1)}{p} \equiv 0 \pmod{p}. \label{eqn:7}
\end{eqnarray}

Now we have that $a$ solves $g^a \equiv a^p \pmod{p^2}$ if and only if the above equivalence holds, and we continue with our proof. 

\ref{sg1is0} $\implies$ \ref{sLifts}:

Assuming \ref{sg1is0}, we have that $\oneunit{g} \equiv 1 \pmod{p^2}$. So then $p^2 \mid (\oneunit{g} - 1)$, and so we get 
\begin{eqnarray*}
\frac{\omega(g)^{x_0}a_0(\oneunit{g}-1)}{p} &\equiv& 0 \pmod{p}. \\
\end{eqnarray*}

Thus all $p$ choices for $a_1 \in \{0,1,\ldots,p-1\}$ give a solution $a \equiv a_0 + a_1p \pmod{p^2}$ to $g^{x} \equiv x^p \pmod{p^2}$ whenever $a \equiv a_0 \pmod{p}$ and $a \equiv x_0 \pmod{m}$. By the Chinese Remainder Theorem, we have exactly one $a \in \{0,1,\ldots,mp^2\}$ where both $a \equiv a_0+a_1p \pmod{p^2}$ and $a \equiv x_0 \pmod{m}$ are satisfied. Since there are $p$ distinct choices for $a_1$, we have $p$ solutions to $g^{x} \equiv x^p \pmod{p^2}$ in $\{0,1,\ldots, mp^2\}$.

\ref{sLifts} $\implies$ \ref{sWieferich}:

Assuming \ref{sLifts}, we know there is at least one $a$ solving
$\frac{\omega(g)^{x_0}a_0(\oneunit{g}-1)}{p} \equiv 0 \pmod{p} $.
 We know by equation (\ref{eqn: 5}) and Fermat's Little Theorem that \\
$\omega(g)^{x_0} \equiv a_0^p \equiv a_0 \pmod{p}$, and since $\omega(g) \not \equiv 0 \pmod{p}$, then $p \nmid a_0$.

Thus we must have $p^2 \mid (\oneunit{g}-1)$, so then 
$$g^{p-1} \equiv \omega(g)^{p-1}\oneunit{g} ^{p-1} \equiv \omega(g)^{p-1}1 ^{p-1} \equiv 1 \pmod{p^2},$$
which means $g$ is a Wieferich base modulo $p$.

\ref{sWieferich} $\implies$ \ref{sg1is0}:

Assuming \ref{sWieferich}, we have that $g^{p-1} \equiv 1 \pmod{p^2}$. Write $\oneunit{g} = 1 + g_1p$. Then we have
$$1 \equiv g^{p-1} \equiv \omega(g)^{p-1}\oneunit{g}^{p-1} \equiv \oneunit{g}^{p-1} \equiv (1+g_1p)^{p-1} \pmod{p^2}.$$
Expanding $(1+g_1p)^{p-1}$ we get
$$ 1 \equiv \sum_{i=0}^{i=p-1}\binom{p-1}{i}(g_1p)^i \equiv 1 + (p-1)g_1p  \pmod{p^2}.$$
So then we have $0 \equiv (p-1)g_1p \pmod{p^2}$, and dividing through by $p$, we obtain
$$(p-1)g_1 \equiv 0 \pmod{p}.$$
Since $p-1 \not \equiv 0 \pmod{p}$, we must have $g_1 \equiv 0 \pmod{p}$.
Thus $$\oneunit{g} \equiv 1 + g_1p \equiv 1 \pmod{p^2}.$$

\end{proof}

The last theorem we give uses our previous lemma as a base case to count solutions to $g^x \equiv x^p \pmod{p^e}$ for $e>1$.

\begin{theorem}
For $p \not = 2$, let $g \in \Zp^\times$ be fixed. Let $N$ be the same as in Theorem \ref{pnmidk_final}. 
Then there are $N$ solutions $x$ to the equation 
$g^{x} \equiv x^p \pmod{p}.$
Furthermore, for $e>1$,  the equation
$g^{x} \equiv x^p \pmod{p^e}$
has $Np$ solutions $x$
if $g^{p-1} \equiv 1 \pmod{p^2}$ (i.e. $g$ is a Wieferich base modulo $p$),
and no solutions otherwise.

\end{theorem}

\begin{proof}
First consider when $e=1$. We have by Lemma \ref{countmodp} that there are $N$ solution pairs $(x_0,x_1) \in \Z/m\Z \times \Z/p\Z$ to $g^{x_0} \equiv x_1^p \pmod{p}$. By the Chinese Remainder Theorem, there is exactly one $x \in \Z/mp\Z$ where $x \equiv x_0 \pmod{m}$ and $x \equiv x_1 \pmod{p}$, so there is exactly one solution $x \in \Z/mp\Z$ where $g^x\equiv g^{x_0} \equiv x_1^p \equiv x^p \pmod{p}$. Since there are $N$ solution pairs, then there are $N$ solutions to $g^x \equiv x^p \pmod{p}$.

Now consider when $e=2$. If $g$ is a Wieferich base modulo $p$, then for each of the $N$ solutions above, we let $a_0$ be a solution and find that 
 we have $p$ solutions to $g^x \equiv x^p \pmod{p^2}$ by Lemma \ref{liftp2}. Since this holds for all our $N$ solutions modulo $p$, we have $Np$ total solutions to $g^x \equiv x^p \pmod{p^2}$.

If $g$ is not a Wieferich base modulo $p$, then by Lemma \ref{liftp2}, none of the $N$ solutions we found modulo $p$ 
lift to a solution to $g^x \equiv x^p \pmod{p^2}$, so there cannot be any solutions modulo $p^2$. Furthermore, there cannot be any solutions to $g^x \equiv x^p \pmod{p^e}$ for $e\geq3$.

When $g$ is a Wieferich base modulo $p$ and $e>1$, we use induction. The base case ($e=2$) is given above, and note that the solutions modulo $p^2$ take the form $a_0 + a_{1}p \pmod{p^2}$, where $a_{1}$ takes any value in $\Z/p\Z$.
Let $$f_{x_0}(x) = F_{x_0}(x)-g^p =\omega(g)^{x_0}\oneunit{g}^{x} - x^p = \omega(g)^{x_0}\Big (\sum_{i=0}^{\infty}\frac{x^i(\log_p{\oneunit{g}})^i}{i!} \Big ) - x^p .$$

For the induction assumption, we assume that we have $Np$ solutions $a \in \Z/mp^{e-1}\Z$ s.t. $g^x -x^k \equiv 0 \pmod{p^{e-1}}$, written
 $a \equiv a_0 + a_1p+ \cdots +a_{e-2}p^{e-2} \equiv a'+a_{e-2}p^{e-2} \pmod{p^{e-1}}$, where $a_{e-2}$ can take any value modulo $p$.

Note that we have $\oneunit{g} \equiv 1 \pmod{p^2}$ by Lemma \ref{liftp2}.

We want to find a solution $x$ that solves $g^x \equiv x^p \pmod{p^e}$, so it must also solve $g^x \equiv x^p \pmod{p^{e-1}}$. So we must have $x \equiv a \pmod{p^{e-1}}$ for one of our solutions $a$. Thus we have that $x \equiv a+a_{e-1}p^{e-1} \pmod{p^e}$ for some $a$. Set $x_0 \equiv a \pmod{m}$.
Note that $f_{x_0}(x) = g^x-x^p$ when $x \equiv x_0 \pmod{m}$ by Theorem \ref{interp}, so $f_{x_0}(a) \equiv 0 \pmod{p^{e-1}}$.

So we have 
\begin{eqnarray*}
f_{x_0}(x) &\equiv& f_{x_0}(a + a_{e-1}p^{e-1}) \\
&\equiv& \omega(g)^{x_0}\Big (\sum_{i=0}^{\infty}\frac{(a+a_{e-1}p^{e-1})^i(\log_p{\oneunit{g}})^i}{i!} \Big ) - (a + a_{e-1}p^{e-1})^p \pmod{p^e}
\end{eqnarray*}

Since for all $ k \in \Z^+,$ $p\mid \frac{(\log_p{\oneunit{g}})^k}{k!}$ (by ~\cite[Lemma 4.5.4]{gouvea}), and because $a \equiv a'+a_{e-2}p^{e-2} \pmod{p^{e-1}}$, we can simplify to get
\begin{eqnarray*}
f_{x_0}(x) &\equiv& \omega(g)^{x_0}\Big (\sum_{i=0}^{\infty}\frac{(a'+a_{e-2}p^{e-2})^i(\log_p{\oneunit{g}})^i}{i!} \Big ) - (a' + a_{e-2}p^{e-2} + a_{e-1}p^{e-1})^p \\
&\equiv& \omega(g)^{x_0}\Big (1+(a'+a_{e-2}p^{e-2})\log_p{\oneunit{g}}+\sum_{i=2}^{e-1}\frac{(a')^i(\log_p{\oneunit{g}})^i}{i!} \Big ) \\
&& \hspace{4cm} - (a' + a_{e-2}p^{e-2} + a_{e-1}p^{e-1})^p \pmod{p^e} \\
\end{eqnarray*}

Note that 
\begin{eqnarray*}
(a' + a_{e-2}p^{e-2} &+& a_{e-1}p^{e-1})^p \\
&\equiv& \sum_{i=0}^{p}\binom{p}{i}(a'+a_{e-2}p^{e-2})^{p-i}(a_{e-1}p^{e-1})^i\\
&\equiv& (a' + a_{e-2}p^{e-2})^p \\
&\equiv& \sum_{i=0}^{p}\binom{p}{i}(a')^{p-i}(a_{e-2}p^{e-2})^i\\
&\equiv& (a')^p + (a')^{p-1}a_{e-2}p^{e-1} \pmod{p^e}.
\end{eqnarray*}

So now we have that, given a solution $x \equiv a \pmod{p^{e-1}}$, if it lifts to a solution modulo $p^e$, then it lifts to any $x \equiv a+a_{e-1}p^{e-1}$.
When we set $f_{x_0}(x) \equiv 0 \pmod{p^e}$ to solve for $x$, we get 
\begin{eqnarray*}
0 \equiv \omega(g)^{x_0}\Big (1+(a'+a_{e-2}p^{e-2})\log_p{\oneunit{g}}+\sum_{i=2}^{e-1}\frac{(a')^i(\log_p{\oneunit{g}})^i}{i!} \Big )\\
 - ((a')^p + (a')^{p-1}a_{e-2}p^{e-1}) \pmod{p^e},
\end{eqnarray*}
so collecting all the $a_{e-2}$ terms, we find that
\begin{eqnarray*}
-a_{e-2}(\omega(g)^{x_0}p^{e-2}\log_p{\oneunit{g}}&-&(a')^{p-1}p^{e-1}) \\
&\equiv& \omega(g)^{x_0} \Big ( 1+a'\log_p{\oneunit{g}}+ \sum_{i=2}^{e-2}\frac{(a')^i(\log_p{\oneunit{g}})^i  }{i!}\Big )-(a')^p \\
&& \hspace{1cm}+ \frac{\omega(g)^{x_0}(a')^{e-1}(\log_p{\oneunit{g}})^{e-1}}{(e-1)!} \pmod{p^e}.
\end{eqnarray*}
Note that $f_{x_0}(x) \equiv f_{x_0}(a) \equiv 0 \pmod{p^{e-1}}$ and so since $p^{e-1}$ divides the left hand side, $p^{e-1}$ must also divide the right hand side, and we also know $p^{e-1} \mid (\log_p{\oneunit{g}})^{e-1}).$
So we can write
\begin{eqnarray*}
&&a_{e-2}\frac{-(\omega(g)^{x_0}p^{e-2}\log_p{\oneunit{g}}-(a')^{p-1}p^{e-1})}{p^{e-1}}\\
&& \hspace{2.5cm}\equiv \frac{\omega(g)^{x_0} ( 1+a'\log_p{\oneunit{g}}+ \sum_{i=2}^{e-2}\frac{(a')^i(\log_p{\oneunit{g}})^i  }{i!} )-(a')^p}{p^{e-1}} \\
&&\hspace{5cm}+ \frac{\omega(g)^{x_0}(a')^{e-1}(\log_p{\oneunit{g}})^{e-1}}{(e-1)!p^{e-1}} \pmod{p}\\
\end{eqnarray*}

It suffices to show that $\frac{-(\omega(g)^{x_0}p^{e-2}\log_p{\oneunit{g}}-(a')^{p-1}p^{e-1})}{p^{e-1}} \not \equiv 0 \pmod{p}$, since doing so would mean it has an inverse modulo $p$ and we can solve uniquely for $a_{e-2}$. Then we would know that exactly one out of every $p$ solutions $a \equiv a'+a_{e-2}p^{e-2} \pmod{p^{e-1}}$ lifts to a solution modulo $p^e$.

If we assume it is congruent to $0$ modulo $p$, we have 
\begin{eqnarray*}
\frac{-(\omega(g)^{x_0}p^{e-2}\log_p{\oneunit{g}}-(a')^{p-1}p^{e-1})}{p^{e-1}} \equiv 0 \pmod{p} 
\end{eqnarray*}
so 
\begin{eqnarray*}
\frac{-(\omega(g)^{x_0}p^{e-2}\log_p{\oneunit{g}})}{p^{e-1}} \equiv (a')^{p-1} \equiv 1  \pmod{p}
\end{eqnarray*}
by Fermat's Little Theorem. But 
\begin{eqnarray*}
\frac{-(\omega(g)^{x_0}p^{e-2}\log_p{\oneunit{g}})}{p^{e-1}} \equiv \frac{-\omega(g)^{x_0}\log_p{\oneunit{g}}}{p} 
\equiv \frac{\omega(g)^{x_0}(\sum_{i=1}^{\infty} \frac{(-1)^{i+1}(\oneunit{g}-1)^i}{i})}{p} \pmod{p}
\end{eqnarray*}
and $\oneunit{g} \equiv 1 \pmod{p^2}$ so we get
\begin{eqnarray*}
\frac{\omega(g)^{x_0}(\sum_{i=1}^{\infty} \frac{(-1)^{i+1}(\oneunit{g}-1)^i}{i})}{p}  \equiv 0 \pmod{p}.
\end{eqnarray*}
This is a contradiction since $1 \not \equiv 0 \pmod{p}$. Thus we can solve uniquely for $a_{e-2}$, and for such an $a_{e-2}$, any $a_{e-1} \in \Z/p\Z$ solves $f_{x_0}(a+a_{e-1}p^{e-1}) \equiv 0 \pmod{p^e}$.

By our induction assumption, there are $Np$ solutions to $f_{x_0}(a) \equiv 0 \pmod{p^{e-1}}$, $a \equiv a'+a_{e-2}p^{e-2} \pmod{p^{e-1}}$. Since for each solution $a'$ to $f_{x_0}(a') \equiv 0 \pmod{p^{e-2}}$, $ a'+a_{e-2}p^{e-2}$ is a solution modulo $p^{e}$ for a unique $a_{e-2} \in \Z/p\Z$, we have exactly $N$ distinct $a$ that lift to $p$ solutions to $f_{x_0}(x) \equiv f_{x_0}(a+a_{e-1}p^{e-1}) \equiv 0 \pmod{p^{e}}$. This gives a total of $Np$ solutions modulo $p^e$.

\end{proof}

Again, we see that this is consistent with our results. For an example, we return to our results  when $p=11$ detailed in Table \ref{table2}.
If we look at just $g=3$ and $g=4$, we have that $N=5$ and so $Np=55$. We find that $3^{10} \equiv 1 \pmod{11^2}$  and $4^{10} \nmid 1 \pmod{11^2}$, so $g=3$ is a Wieferich base and has $55$ solutions modulo $11^e$ for $e >1$, which $g=4$ is not a Wieferich base modulo $11$ and thus has no solutions modulo $p^e$ for $e>1$. Both of these match our findings in Table \ref{table2}.

\section{Conclusions and Future Work}
In this paper, we have applied the methods found in \cite{hold_rob}, \cite{mann_yeoh}, and \cite{waldo_zhu} to count solutions to the equation $g^{x^n} \equiv x^k \pmod {p^e}$.  When analyzing the equation for $x \in \{1,2, \ldots, mp^e\}$, $p$ an odd prime, we have found an exact number of solutions for the case when $p \nmid k$, specifically $N = \frac{m\cdot\gcd(k,p-1)}{q_1^{\lceil{\frac{\alpha_1}{n}}\rceil}q_2^{\lceil{\frac{\alpha_2}{n}}\rceil}\cdots q_i^{\lceil{\frac{\alpha_i}{n}}\rceil}}$ solutions. In addition, we found that when $p$ odd, $k=p$ and $n=1$, that there are $N$ solutions when $e=1$, and either $Np$ or $0$ solutions when $e>1$, depending on whether $g$ is a Wieferich base modulo $p$.

It remains to be shown whether the same number of solutions is obtained for general $n$ in the $k=p$ case. We suspect that similarly to the $p \nmid k$ case, $n$ will only affect the value of $N$, and not the results of lifting solutions modulo $p$ to solutions modulo $p^e$, but this has not been confirmed. Additionally, the case where $p \mid k$ but $k \not = p$ has not been analyzed yet. Based on a few test results, we suspect that $k$ will affect the value of $N$, but that the number of solutions modulo $p^e$ for $e>1$ will be the same as in the $p=k$ case. Lastly, due to the fact that we need $x \in 1+4\Z_2$ rather than $x \in 1+2\Z_2$ for $\log_2(\exp_2(x))$ to converge, the case where $p=2$ must be analyzed differently. We have done some testing that confirms that $p=2$ yields different results than for odd prime $p$, leaving another avenue of analysis. 

Besides counting solutions, further analysis can be done by exploring the distribution of solutions for $x$ among intervals of length $p^e$ rather than focusing on only the longer interval of length $mp^e$, since that range is generally more applicable to cryptographic schemes. There has been some analysis of the map $x \mapsto g^{x^n} \pmod{c}$ when $n=2$ by Wood \cite{wood}, and some statistical analysis of the map when $n=1$ from \cite{lindle} and \cite{cloutier}, but more work remains to be done.


\end{document}